\documentclass[12pt]{amsart}
\usepackage{tikz}
\usepackage{caption}  
\usepackage{amssymb,amsmath,latexsym,amsthm}
\newtheorem{thm}{Theorem}[section]
\newtheorem{lemma}[thm]{Lemma}

\newtheorem{cor}[thm]{Corollary}

\theoremstyle{definition}
\newtheorem{defn}[thm]{Definition}
\newtheorem{conj}[thm]{Conjecture}

\numberwithin{equation}{section}
\newtheorem{claim}{Claim}
\begin{document}

\title[Infinitely many sign changes of the Liouville function]{ Infinitely many sign changes of the Liouville function on $x^2+d$
}
\author{Anitha  Srinivasan}
\address{Department of Mathematics, Saint Louis University-Madrid campus, 
Avenida del Valle 34, 28003 Madrid, Spain}

\date{}
\maketitle

\begin{abstract}
We show that the Liouville function $\lambda$ changes sign infinitely often on $n^2+ d$ for any non-zero integer $d$. 

\end{abstract}

\bigskip

 \begin{section} {Introduction }\end{section}

For a given non-zero integer $d$, is $n^2+d$  prime for infinitely many integers  $n$? While hardly anyone would 
doubt that the answer here is a yes, no one has been able to prove it yet! In order to make some progress towards the answer, one could modify the question to ask instead, whether 
$n^2+d$ has an odd number of prime divisors (counting multiplicities) infinitely often. In this work we answer this latter question in the affirmative.

The Liouville function is defined as
$\lambda(n)=(-1)^{\Omega(n)}$, where 
$\Omega(n)$ denotes the number of prime factors of
the integer $n$ counted with multiplicity. In other words,
$\lambda$ is a completely multiplicative function,
where $\lambda(p)=-1$ for each prime $p$ 
and $\lambda(1)$ is taken as $1$.

Cassaigne et al. in \cite{C} made the following conjecture.
\begin{conj}{\label C}(Cassaigne et al.)
If $f(x)\in \mathbb{Z}[x]$ and
$f(x)\neq b(g(x))^2 $ for any
integer $b$  and
$g(x)\in \mathbb{Z}[x]$, then 
$\lambda(f(n)$ changes sign infinitely often.
\end{conj}

This was a weaker form of the following conjecture made by Chowla
  in \cite[p. 96]{Ch}. 

\begin{conj}(Chowla)
Let $f(x)$ be an arbitrary  polynomial with integer coefficients,  
which is not, however, of the form 
$c (g(x))^2 $ where $c$ is an  
integer  and 
$g(x)$ is a polynomial with integer coefficients. Then  
$$\sum_{n=1}^{x} \lambda(f(n))=o(x).$$
\end{conj}
Chowla also added the comment: 
{\noindent \lq\lq} If 
$f(x)=x$ this is equivalent to the   
Prime Number Theorem. If the degree of $f(x)$ is at least $2$, this 
seems an extremely hard conjecture."

Conjecture \ref{C} though far weaker than Chowla's conjecture, has been proved only for some special cases. Several  cases of specific quadratic polynomials have been dealt with by
Cassaigne et al. \cite{C},  Borwein, Choi and Ganguly \cite{BCG} and Srinivasan \cite{S}.  Ter{\"{a}}v{\"{a}}inen \cite{T} proved Conjecture \ref{C} for 
polynomials that factor into linear factors, for those with certain special quadratic factors, and for reducible cubics.  

 Borwein, Choi and Ganguly \cite{BCG} gave a useful algorithm to show that the 
Liouville function changes sign infinitely often for
any  given quadratic polynomial, hence proving the conjecture for a fixed quadratic polynomial. However, this algorithm cannot be used for arbitrary quadratic polynomials.  

In our main theorem below we settle Conjecture \ref{C} for all the quadratic polynomials $x^2+d$. 

\begin{thm}
Let $f(x)=x^2+d$ where $d\ne 0$. Then 
$\lambda(n^2+ d)$ changes sign infinitely often.
\end{thm}
Our approach uses the theory of genera of binary quadratic forms. For every discriminant there is a set of assigned characters that determine the genera of the forms.
  Each form belongs to a genus depending on its associated generic values.  Now consider the form $n^2+d$ with $d>0$. Our main tool is the construction of a positive integer $M$ that satisfies $\lambda(M)=-\lambda(d)$, and such 
that  $(M, 0, -d)$ is the only ambiguous form (with first coefficient greater than $1$) in the principal genus. We also ensure that the norm of the fundamental unit corresponding to the discriminant $dM$ is positive. The positivity of this norm implies that there is an ambiguous form other than the identity form in the principal class and hence in the principal genus. As mentioned above, our choice of $M$ forces this form to be $(M, 0, -d)$, giving us infinitely many solutions $(x, l)$ to 
$Mx^2-dl^2=1$ (all forms in the principal class represent the integer $1$). 
If we now take $n=dl$, then we have 
$\lambda(n^2+d)=\lambda(d(dl^2+1))=\lambda(dMx^2)=-1$. It follows that the negative sign
is attained infinitely often by $ \lambda(n^2+d)$ and this is sufficient to conclude that there 
are infinitely many sign changes (Lemma 3.1).

We remark here that our method does not work for quadratic polynomials not of the 
form $x^2+d$.  For a general quadratic polynomial $ax^2+bx+c$ of discriminant $d$, we may re-write $ax^2+bx+c=m$ 
 as
$(2ax+b)^2-d=4am$. Observe that $X=2ax+b$ and $b$ are of the same parity (as are $b$ and $d$). Furthermore, if $a>1$ or $b\ne 0$, then we require  
 $X\equiv b\pmod {2a}$, so that in the approach above, the integers $l$ coming from 
 the solutions to $Mx^2-dl^2=1$ need to satisfy certain conditions. However in our case of
$x^2+d$, there are no restraints on the integers $l$. 
\begin{section} {  The theory of genera }

Let 
$d\equiv 0\text{ or }1 \mod 4$ be an integer that is not a
square 

A {\sl primitive binary quadratic form} $f=(a, b, c)$ of discriminant $d$ is
a function $f(x, y)=ax^2+bxy+cy^2$, where $a, b,c $ are integers
with $b^2-4 a c=d$ and $\gcd(a, b, c)=1$. All forms considered here
are primitive binary quadratic forms and henceforth we shall
refer to them simply as forms. 

Two forms $f$ and $f'$ are said to be {\it equivalent}, written as
$f\sim f'$,  if for some
$A=\begin{pmatrix} \alpha &\beta \\ \gamma & \delta \end{pmatrix}
\in SL_2(\mathbb Z)$ we have
$f'(x,y)=f(\alpha x+\beta y, \gamma x+\delta y)$. The equivalence classes form an abelian group
called the  {\it class group} with group law given by composition of
forms.

The {\it identity form} is defined as the form $(1,0,\frac{-d}{4})$
or $(1, 1, \frac{1-d}{4})$ depending on whether $d$ is even or odd,
respectively. The class represented by the identity form is 
called the {\it principal} class.
 The {\it inverse} of
$f=(a, b, c)$, denoted by $f^{-1}$ is given by $(a,-b,c)$ or $(c, b, a)$.

A form $f$ is said to represent an integer $m$ if there exist coprime
integers $x$ and $y$ such that $f(x,y)=m$. It is easy to see that 
equivalent forms represent the same integers, as do a form $f$ and 
its inverse $f^{-1}$.

A form  $(a, b, c)$ is said to be {\it ambiguous } if $a$ divides $b$.
It follows then that $a$ divides $d$. 
 Ambiguous classes have orders $1$ or $2$ in the class group. 


The  assigned characters for all discriminants is given in \cite[Page 138]{R}. Below, we list these characters
only for  discriminants $4d$ with $d$ square-free, as this is the case we are interested in.
\begin{defn}{(Assigned characters)}
Let $D>0$ be a square-free integer and let 
$r_1, r_2\cdots r_t$ be 
its distinct odd prime divisors. Let $m$ be an integer such that 
$\gcd(m, 2D)=1$.
   For each $1\le i\le t$, let $\chi_i(m)=(\frac{m}{r_i})$ be a character modulo $r_i$. 
Let 
$\delta(m)=(-1)^{\frac{m-1}{2}}$ and 
$\eta(m)=(-1)^{\frac{m^2-1}{8}}$ be characters modulo $4$ and $8$ respectively. The set of {\it assigned characters} associated
to $D$ are given by the following rule:
\begin{table}[ht]
\center
\begin{tabular}{c| c |}
Discriminant & Assigned characters\\
$D=4d, d\equiv 1\pmod 4$ & $\chi_1, \cdots \chi_t$\\
$D=4d, d\equiv 3\pmod 4$ & $\chi_1, \cdots \chi_t, \delta$\\
$D=4d, d\equiv 2\pmod 8$ & $\chi_1, \cdots \chi_t, \eta$\\
$D=4d, d\equiv 6\pmod 8$ & $\chi_1, \cdots \chi_t, \delta\eta$
\end{tabular}
\end{table}
\end{defn}
Let $f$ be a form of discriminant $D$ and  $m$ with $\gcd(2d, m)=1$ be any integer
represented by $f$. Then the generic values assigned to $f$ are the 
values of the assigned characters  at $m$.
The {\it principal genus} consists of all forms with all generic values equal to $1$. Clearly 
the principal class belongs to the principal genus, as it represents the integer $1$.
It can be shown that the generic values for a given class of forms is
independent of the integer $m$ represented.

The subject of the following lemma is a particular factorization 
of the discriminant that corresponds to the principal class, in the case when the fundamental unit has positive norm.  The lemma is a slightly altered version of  \cite[Proposition 3.1]{MS}, where we have added a remark that is pertinent to the problem at hand.
\begin{lemma}{\rm (\cite[Proposition 3.1]{MS})}
\label{norm}
Let $d$ be a positive integer that is not a square. Suppose that the norm 
of the fundamental unit of the quadratic field $\mathbb{Q}(\sqrt{d})$ 
is positive. Then one of the following occurs.
\begin{enumerate}
\item
There exists a factorization $d=ab$ with
 with $1<a<b$ such that one of the following equations has an integral  solution $(x,y)$.
$$ax^2-by^2=\pm 1.$$ 
\item
  There exists a factorization $d=ab$ with $1\le a<b$  and $d\not\equiv 1, 2\pmod 4$,
such that one of the following equations has an integral  solution $(x,y)$ with $xy$ odd.
\end{enumerate}   
$$ ax^2-by^2=\pm 2.$$
Furthermore, for $d=ab$, exactly one of the following ambiguous forms 
 is in the principal class, where the first set corresponds to part 1 above,
and the second to part 2:
\begin{equation}\label{normform}
(a, 0, -b) {\text{ with }} 1<a, {\text { or }}  
\left(2a, 2a, \frac{a-b}{2}\right) 
 {\text{ with }} 1\le a.
\end{equation}
\end{lemma}
\begin{proof}
Parts 1 and 2 are given in \cite[Proposition 3.1]{MS}. If $ ax^2-by^2= 1$ holds, then clearly the form $(a, 0, -b)$ represents $1$ and hence is in the principal class. Similarly, if $ ax^2-by^2= -1$, then $(b, 0, -a)$ is in the principal class. 

Now consider $ax^2-by^2= 2$. Taking 
$(\alpha, \beta)=(\frac{x-y}{2}, y)$ we may verify that 
$2a\alpha^2+2a\alpha\beta+\frac{a-b}{2}\beta^2=1$. Hence the form 
$\left(2a, 2a, \frac{a-b}{2}\right)$ is in the principal class. The equation with 
the negative sign would put the form 
$\left(2b, 2b, \frac{b-a}{2}\right)$
 in the principal class.
\end{proof}
\end{section}
\begin{section} {Construction of the integer $M$}
We begin this section with the following lemma,  stated a little differently from its original version, where only one
negative sign for $\lambda(f(n))$ was required.
\begin{lemma}{\rm (\cite[Theorem 2.4]{BCG})}
Let $f(x)=x^2+bx+c$ with $b, c\in \mathbb{Z}$.
Then if $\lambda(f(n))$  is equal to $-1$ infinitely often, then 
it changes sign infinitely many times.
\end{lemma}
\begin{lemma} Let $d>0$ be an integer that is not a square.
There are infinitely many integers $n$ such that 
$\lambda(n^2+d)=\lambda(d)$.
\end{lemma}
\begin{proof}
It is well known that the Pell equation 
$x^2-dy^2=1$ has infinitely many solutions $(x, y)$.  For a solution $(r, l)$, let 
$m=d+d^2l^2=d(1+dl^2)=dr^2$. Observe that the form 
$(m, 0, -1)$ represents $d$ via the representation 
$(1, dl)$. Thus there are infinitely many representations $(\alpha, \beta)$ satisfying
$m\alpha^2-\beta^2=d$, so that we may choose $\beta$ as large as we please. 
We have $\beta^2+d=m\alpha^2$ from which it follows that 
$\lambda(\beta^2+d)=\lambda(m)=\lambda(dr^2)=\lambda(d)$, and the lemma follows.
\end{proof}
\begin{lemma}
Let $d$ be a square-free integer. If 
$\lambda(n^2+d)$ changes sign infinitely often, then 
 $\lambda(n^2+dt^2)$ also changes sign infinitely often for any integer $t$.
\end{lemma}
\begin{proof}
The claim follows from the observation that $\lambda(t^2(n^2+d))=\lambda((nt)^2+dt^2)=\lambda(n^2+d)$.
\end{proof}
\begin{lemma}
If $d=p$ is prime, then $\lambda(n^2\pm d)$ changes sign infinitely often.
\end{lemma}
\begin{proof}
A well known fact is that if $p\equiv 1\pmod 4$ then both Pell equations $y^2-pk^2=\pm 1$ have infinitely many solutions.
 Let $n=dk$. 
Then $n^2\pm d=d(dk^2\pm 1)=dy^2$ and thus 
$\lambda(n^2\pm d)=\lambda(d)=-1$ and the result follows from Lemma 3.1. Observe that in the case of
$n^2+d$ we used the positive Pell equation, which holds for any $p$. Hence it remains to consider the case 
$n^2-p$ with $p\equiv 3\pmod 4$.

In this case the fundamental unit has positive norm. Using the Chinese remainder theorem and 
Dirichlet's theorem of primes in arithmetic progressions, we can find 
primes $e_1\equiv 3\pmod 4$ and 
$e_2\equiv 1\pmod 4$ that satisfy the following:
\begin {enumerate}
\item 
$\left(\frac{p}{e_1}\right)=
\left(\frac{p}{e_2}\right)=1$.
\item 
$\left(\frac{e_1}{e_2}\right)=-1.$
\end{enumerate}
Let $m=e_1e_2$. Then 
 Lemma 2.2 for $d=mp$ applies; in particular, part 1 holds. Thus one of
 the forms $(a, 0, -b)$, where $ab=mp$ is in the principal class, and hence in the 
principal genus. We will show that the only such form in the principal genus is
$(p, 0, -m)$. To determine the generic characters of the form $(a, 0, -b)$ we will use 
the integers $\theta_1=4a-b$ or $\theta_2=a-4b$ that satisfy $\gcd(\theta_1\theta_2, 2mp)=1$ and are
represented by the form in question.
We will work with the generic characters  modulo $p, e_1$ and $e_2$.

When $a=e_1$ (and $b=pe_2$) we have 
$\left(\frac{\theta_2}{p}\right)=
\left(\frac{e_1}{p}\right)=-1$. Similarly for the forms 
with $a=pe_1$ and $a=e_1e_2$, we have 
$\left(\frac{\theta_1}{p}\right)=
\left(\frac{-e_2}{p}\right)=-1$, 
and 
$\left(\frac{\theta_2}{p}\right)=
\left(\frac{e_1e_2}{p}\right)=-1$
respectively. For the forms with 
$a=e_2$ and $a=pe_2$ we consider the character modulo $e_2$, to get
$\left(\frac{\theta_1}{e_2}\right)=
\left(\frac{-pe_1}{e_2}\right)=-1$ for the former and 
$\left(\frac{\theta_1}{e_2}\right)=
\left(\frac{-e_1}{e_2}\right)=-1$ for the latter. 

We have eliminated above all forms $(a, 0, -b)$ with $a\ne p$ from the 
principal genus (and hence from the principal class).
Thus the only form from (\ref{normform}) in the principal class is $(p, 0, -m)$ and therefore
we have $pk^2-my^2=1$, and as before, 
  taking $n=pk$ we get 
$n^2-p=p(pk^2-1)=pmy^2$ yielding
$\lambda(n^2-p)=\lambda(pm)=-1$ and again the result follows from Lemma 3.1.

\end{proof}
\begin{thm}\label{m}
Let $d=p_1 p_2\cdots p_r$, where $r\ge 2$ and $p_i$ are primes with 
 $p_r=2$ if $d$ is even.  Let 
$m=m_1m_2\cdots m_{r-1}$, where $m_i$ are primes. 
If $r=2$, then $m_1\equiv 5 \pmod 8$. For $r\ge 3$,
$m_i\equiv 1\pmod 8$ for $1\le i\le r-2 $ and 
$m_{r-1}\equiv 5\pmod 8$. Let $e_1, e_2$ be primes that are congruent to $-1$ and $3t\pmod 8$ respectively,
where $t=1$ or $-1$. Moreover, if 
$\lambda(d)=-1$, then $t=-1$.
  Suppose that the Legendre symbols of $p_i$ modulo 
 $m_i$ and $e_i$ 
are as given in the following table, where  the plus and minus signs stand 
for $1$ and $-1$, and $s=-t$.
\begin{table}[ht]
\center
\begin{tabular}{c| c |c| c| c| c| c| c| c|c|}
$(\frac{p_j}{*})$  & $m_1$ & $m_2$ &  $m_3$ & $\cdots$ & $m_{r-2}$ & $m_{r-1}$ & $e_1$ & $e_2$ \\ 
\hline
$p_1$ &  $-$ & $- $ &  $- $ & $\cdots$ & $- $ & $- $ & $\lambda(d)s$ & $t$ \\ 
$p_2$ &  $- $ & $ +$ & $+ $ & $\cdots$ & $ +$ & $ +$ & $\lambda(d)s$ & $\lambda(d)t$ \\ 
$p_3$ &  $ +$ & $- $ & $ +$ & $ +$ & $\cdots$ & $ +$ & $\lambda(d)s$ & $\lambda(d)t$ \\ 
$p_4$ &  $ +$ & $ +$ & $- $ & $ +$ & $\cdots$ &  $ +$ & $\lambda(d)s$ & $\lambda(d)t$ \\ 
$\vdots$  &  & & & & $\cdots$ & & & \\ 
$p_{r-1}$ &  $ +$ & $ +$ & $ +$ & $\cdots$ &  $-$  & $ +$ & $\lambda(d)s$ & $\lambda(d)t$ \\ 
$p_r$  & $+$ & $+$ & $+$ & $\cdots$ &  $+$ &  $-$ & $+$ & $-$ \\ 
\end{tabular}
\caption{\label{Table 1}}
\end{table}

Additionally, assume that the primes $m_i$ and $e_i$ satisfy the following.
\begin{enumerate}
\item
$\left(\frac{m_i}{e_j}\right)=1 {\text{ for all }} i, j 
{\text{ and }} 
\left(\frac{m_i}{m_j}\right)=1 {\text{ for all }} i \ne j
.$
\item
$\left(\frac{e_1}{e_2}\right)=t,
\left(\frac{e_2}{e_1}\right)=-1.
$
\end{enumerate}

Then for $D=mde_1e_2=ab$, of all the  forms  given in (\ref{normform}), 
the only one that is in the principal genus is,  $(me_1e_2, 0, -d)$ if $t=1$ and  $(d, 0, -me_1e_2)$ if $t=-1$.
\end{thm}
\begin{proof}
We first note the following for all $i$,  from the table in the hypotheses, where the first observation follows from looking at the columns of the table and  the second one, the rows.  \begin{equation}
\left(\frac{d}{m_i}\right)=1,  
\left(\frac{d}{e_i}\right)=s.
\end{equation}
\begin{equation}
\left(\frac{p_i}{me_1e_2}\right)=1.
\end{equation}

In the case when $d$ is odd, the generic characters include the Legendre symbols 
modulo the primes $m_i$, $p_i$ and $e_i$. When 
$D\equiv 3 \pmod 4$ we also have the character modulo $4$. For even $d$, additionally we have 
the characters $\eta$ or $\delta\eta$ depending on whether $D\equiv 2$ or $6 \pmod 8$. Recall that 
to determine the genus of a form we compute the generic 
values of an integer $\theta$ represented by the form such that
$\gcd(\theta, 2D)=1$.

\begin{claim}
The forms
 $(me_1e_2, 0, -d)$ when $t=1$ and 
 $(d, 0, -me_1e_2)$ when $t=-1$ 
 are in the principal genus. 
\end{claim}
Let $$\theta_1=t(4me_1e_2-d),\hskip2mm   
\theta_2=t(me_1e_2-4d).$$ 
Note that when $t=1$ both $\theta_1$ and $\theta_2$ are represented
by the form $(me_1e_2, 0, -d)$, and when $t=-1$ both are represented by
$(d, 0, -me_1e_2)$.  We first consider $d$ odd.
Then $(\theta_1\theta_2, 2D)=1$, so we may use these integers to 
show that all the generic characters 
have value $1$ for the forms in consideration.

From equation (3.2) we have 
$$\left(\frac{\theta_2}{p_i}\right)=
\left(\frac{t}{p_i}\right)
\left(\frac{me_1e_2}{p_i}\right)=1,$$ 
noting that 
$me_1e_2\equiv t \pmod 8$.

Next we have 
$$\left(\frac{\theta_1}{m_i}\right)=
\left(\frac{t}{m_i}\right)
\left(\frac{-d}{m_i}\right)=\left(\frac{d}{m_i}\right)=1$$ using (3.1) 
and as $m_i\equiv 1\pmod 4$.

Finally 
$$\left(\frac{\theta_1}{e_i}\right)=
\left(\frac{t}{e_i}\right)
\left(\frac{-d}{e_i}\right)=
\left(\frac{t}{e_i}\right)
\left(\frac{-1}{e_i}\right)s$$
using (3.1) again.
The above value is $1$ if $t=-1$ (and hence $s=1$). If $t=1$ the value is again $1$, 
as $s=-1$ and 
$e_1\equiv e_2\equiv 3\mod 4$.

If $d$ is even then 
$\theta_2\equiv tme_1e_2\equiv  t^2 \pmod 8$ and hence the values for the generic characters
$\eta$ and $\delta$ are both $1$. 

When $D\equiv dt\equiv 3\pmod 4$ we also have the character $\delta$  modulo $4$.
Here $\theta_1\equiv -dt\equiv 1\pmod 4$, and hence 
$\delta(\theta_1)=1$. As all the generic values are equal to $1$, the claim follows.

Next we will show that none of the other ambiguous forms given in (\ref{normform}) is in the 
principal genus. 
\begin{claim} Of all the forms $(a, 0, -b)$, with $ab=mde_1e_2$ and $a>1$, the only ones that are in the principal genus are the ones given in Claim 1.
\end{claim}
Consider the form 
$F=(R_1d_1, 0, -R_2d_2)$, where
 $$R=R_1R_2=me_1e_2, {\text{ and }} d=d_1d_2.$$ 
When  $d$ is odd let
$$
\theta_1=4R_1d_1-R_2d_2 {\text{ and }}
\theta_2=R_1d_1-4R_2d_2. $$ 
For $d$ even, we choose $\theta_1=\theta_2=R_1d_1-R_2d_2$.
Both $\theta_1$ and $\theta_2$ are integers represented by the form $(R_1d_1, 0, -R_2d_2)$ and 
satisfy $\gcd(\theta_1\theta_2, 2D)=1$ and hence we will use these integers to 
evaluate the generic values of $F$.

Suppose that $m_1|R_1$. If $F$ is in the principal genus then 
$$\left(\frac{\theta_1}{m_1}\right)=
\left(-\frac{R_2}{m_1}\right)
\left(\frac{d_2}{m_1}\right)=
\left(\frac{d_2}{m_1}\right)=1.$$ Similarly if 
$m_1|R_2$
$$\left(\frac{\theta_2}{m_1}\right)=
\left(\frac{d_1}{m_1}\right)=1.$$
As $\left(\frac{d}{m_1}\right)=1$ we have 
$\left(\frac{d_1}{m_1}\right)=
\left(\frac{d_2}{m_1}\right)=1$.
   Looking at the first 
column of Table 1 (in particular the minus signs), we conclude that 
either  
$p_1p_2|d_1$ or $p_1p_2|d_2$.
 Using the same reasoning 
for the prime $m_2$, we obtain that $p_1 p_3$ divides either 
$d_1$ or $d_2$, so that combining with the conclusion above, we obtain that $p_1p_2p_3$
divides either 
$d_1$ or $d_2$. Continuing in this way with all the primes
 $m_3\cdots m_{r-1}$, we obtain that either
$d|d_1$ or $d|d_2$  and therefore
$$F=(R_1d, 0, -R_2) {\text{ or }} (R_1,0, -R_2d).$$

Consider the form $(R_1d, 0, -R_2)$. Then 
$$
\theta_1=4R_1d-R_2 {\text{ and }}
\theta_2=R_1d-4R_2. $$ 

If $e_2|R_1$ and $e_1|R_2$, then 
$F$ is not in the principal genus as 
$$\left(\frac{\theta_1}{e_2}\right)=
\left(\frac{-R_2}{e_2}\right)=
\left(\frac{-1}{e_2}\right)
\left(\frac{e_1}{e_2}\right)
\left(\frac{R_2/e_1}{e_2}\right)=
\left(\frac{-1}{e_2}\right)t=
\left(\frac{-1}{3t}\right)t=-1,$$
where we have used assumptions 1. and 2. from the hypotheses, as well as the 
fact that $e_2\equiv 3t\pmod 8$.
If $e_1e_2|R_1$, again $F$ is not in the principal genus as
$$\left(\frac{\theta_1}{e_1}\right)=
\left(\frac{-R_2}{e_1}\right)=
\left(\frac{-1}{e_1}\right)=-1.$$
Next consider the case when 
 $e_2|R_2$. Suppose that $e_1|R_1$. Then we have 
$$\left(\frac{\theta_2}{e_2}\right)=
\left(\frac{e_1}{e_2}\right)
\left(\frac{d}{e_2}\right)
\left(\frac{R_1/e_1}{e_2}\right)
=ts=-1,$$
using hypotheses 1. and 2. and (3.1). 
Assume now that  $e_1e_2|R_2$.
Then
$$\left(\frac{\theta_2}{e_1}\right)=
\left(\frac{R_1}{e_1}\right)
\left(\frac{d}{e_1}\right)=s.$$
Therefore if $s=-1$ then this form is not in the principal genus.
Suppose that $s=1$ (and hence $t=-1$). 
Then we have $e_1e_2\equiv 3\pmod 8$
and 
$\left(\frac{p_{2}}{e_1e_2}\right)=-1$ and thus 
\begin{equation}{\label{p2}}
\left(\frac{\theta_1}{p_{2}}\right)=
\left(\frac{-R_2}{p_{2}}\right)=
\left(\frac{-e_1e_2}{p_{2}}\right)
\left(\frac{R_2/(e_1e_2)}{p_{2}}\right)
=
-\left(\frac{R_2/(e_1e_2)}{p_{2}}\right). 
\end{equation}
Now if $p_2=2$, then $d=2p_1$ and $m=m_1$, and the forms to consider here are
$(2p_1, 0, -e_1e_2m)$ (in which case the claim holds), and $(2mp_1, 0, -e_1e_2)$.  For the
latter form  consider  
$\theta=8pm-e_1e_2\equiv 5\pmod 8$, which implies that 
$\eta(\theta)=-1$ and $\delta(\theta)=1$, and therefore both characters 
$\eta$ and $\delta\eta$ have value $-1$, and hence the form is not in the principal genus.
We may assume now that $p_2$ is odd.

If $F$ is in the principal genus then from (\ref{p2}) we have 
$$
\left(\frac{R_2/(e_1e_2)}{p_{2}}\right)=
\left(\frac{p_{2}}{R_2/(e_1e_2)}\right)=
-1.$$
Looking at the second row in Table 1 we infer that 
$m_{1}|R_2$, as we must choose the first minus sign in the row
corresponding to $p_2$. Using an identical reasoning with the 
primes $p_{3}, p_4\cdots p_{r-1}$ we obtain 
that $\frac{m}{m_{r-1}}|R_2$. Now in the case when $d$ is odd,
we may also consider the prime $p_r$, to obtain $m_r|R_2$ and 
hence  $m|R_2$, yielding the claimed form 
$(d, 0, -me_1e_2)$. In the case when $d$ is even, as $p_r=2$, we 
have the  characters corresponding to
either $\eta$ or $\delta\eta$. Consider $\theta=4R_1d-R_2$. Then 
$\theta\equiv -\frac{me_1e_2}{m_{r-1}}\equiv 5 \pmod 8$ ($t=-1$). Hence 
both the characters  
$\eta$ or $\delta\eta$ have value $-1$ and thus
$F$ is not in the principal genus.

Now we turn to the form 
$(R_1,0, -R_2d)$. In this case we have
$$
\theta_1=4R_1-R_2 d {\text{ and }}
\theta_2=R_1-4dR_2. $$ 

 Suppose $e_1|R_1$ and $e_2|R_2$. Then using (3.1), 1. and 2. from the hypotheses we have
$$\left(\frac{\theta_1}{e_1}\right)=
\left(\frac{-R_2d}{e_1}\right)=
-\left(\frac{e_2}{e_1}\right)
\left(\frac{R_2/e_2}{e_1}\right)
\left(\frac{d}{e_1}\right)
=s.$$
Also
$$\left(\frac{\theta_2}{e_2}\right)=
\left(\frac{R_1}{e_2}\right)=
\left(\frac{e_1}{e_2}\right)
=t.$$ 
Therefore one of $\left(\frac{\theta_1}{e_1}\right)$ or
$\left(\frac{\theta_2}{e_2}\right)$ is $-1$ and hence the 
form in consideration is not in the principal genus.
  In the case when $e_1|R_2$ and $e_2|R_1$  we get 
$$\left(\frac{\theta_2}{e_1}\right)=
\left(\frac{R_1}{e_1}\right)=
\left(\frac{e_2}{e_1}\right)=
-1.$$ Hence we now
assume that either $R_1$ or $R_2$ is divisible by $e_1e_2$. 

Suppose that $R_1$  is divisible by $e_1e_2$. Note that $e_1e_2\equiv t\pmod 4$.
If $t=-1$ then the form is not in the principal genus as
$$\left(\frac{\theta_1}{e_1e_2}\right)=
-\left(\frac{R_2d}{e_1e_2}\right)=
-s^2$$ 
When $t=1$, we  have for all $i$ 
$$\left(\frac{e_1e_2}{p_{i}}\right)=-1$$ 
and thus 
$$\left(\frac{\theta_2}{p_{i}}\right)=
\left(\frac{R_1}{p_{i}}\right)=
\left(\frac{e_1e_2}{p_{i}}\right)
\left(\frac{R_1/(e_1e_2)}{p_{i}}\right)=
-\left(\frac{R_1/(e_1e_2)}{p_{i}}\right).
$$
To be in the principal genus the above value must be equal 
to $1$ for all odd $p_i$, which means that 
$\frac{R_1}{e_1e_2}\ne 1$. If $r=2$ ($m=m_1$), it follows that 
$R_1=me_1e_2$ and we have the desired form in the principal genus. We may assume now that 
$r\ge 3$.

If $F$ is in the principal genus then for the generic
character corresponding to $p_{r-1}$,  we have 
$\left(\frac{R_1/e_1e_2}{p_{r-1}}\right)
=-1$,
which means that 
$m_{r-2}|R_1$ (we need to choose the minus sign corresponding to
$m_{r-2}$ in the last but one row). Using the exact same reasoning with each prime 
$p_{r-2},\cdots p_1$ we arrive at the fact that $\frac{m}{m_{r-1}}|R_1$. In the case when 
$d$ is odd, we also have that $m_{r-1}|R_1$ (using the last row) and hence our form is
$(me_1e_2, 0, -d)$ and the claim is true.

In the case when $d$ is even, we have the form 
$(R/m_{r-1}, 0, -m_{r-1}d)$. Consider 
$\theta=R/m_{r-1}-4m_{r-1}d$.
Then
$\theta\equiv R/m_{r-1}\equiv 5 \pmod 8$ ($t=1$) which means that the characters 
$\eta(\theta)$ and $\eta\delta(\theta)$ are both $-1$ and hence
the form in question is not in the principal genus.

It remains to consider the case when 
$R_2$  is divisible by $e_1e_2$. If the form is in the principal genus, then for
all odd $p_i$ we have 

$$\left(\frac{\theta_2}{p_{i}}\right)=
\left(\frac{R_1}{p_{i}}\right)=
\left(\frac{p_i}{R_{1}}\right)=1.$$

Reasoning as above for each $i=r, r-1, \cdots 2$ (looking at the $i^{\rm th}$ row), we get
that $m|R_2$ if $d$ is odd, giving the form
 $(1, 0, -Rd)$ (which is not in the list (\ref{normform})). For
$d$  even as before we get the form 
$(m_{r-1}, 0, -Rd/m_{r-1})$, and  
since $m_{r-1}\equiv 5 \pmod 8$, as argued above, we
conclude that the form is not in the principal genus.
\begin{claim} The forms $(2a, 2a, \frac{a-b}{2})$ for $me_1e_2d=ab$ 
are not in the principal genus.
\end{claim}
Note from Lemma 2.2 that this case occurs only
when $ab=me_1e_2d\equiv 3\pmod 4$.
For the forms in question we consider 
$\theta=\frac{a-b}{2}$.  
We have
$$
\left(\frac{\theta}{m_{r-1}}\right)=
\left(\frac{4\theta}{m_{r-1}}\right)=
\left(\frac{2(a-b)}{m_{r-1}}\right)=
-\left(\frac{a-b}{m_{r-1}}\right)=-1
$$ 
since $m_{r-1}\equiv 5 \pmod 8$ and as $m_{r-1}|a$ or $m_{r-1}|b$, we have
 $\left(\frac{a-b}{m_{r-1}}\right)=1$ (using (3.1) and statement 1. from the hypotheses).

 \end{proof}

\begin{cor} Let $d$ be a square-free positive integer that is not prime. Then there exists a square-free positive integer $M$ with $\lambda(M)=-\lambda(d)$ that satisfies one of the following.
\begin{enumerate}
\item
If $\lambda(d)=1$ then there 
are infinitely many  integer solutions $(y, k)$  and $(x, n)$ to 
$My^2-dk^2=1$ and $dx^2-Mn^2=1$.
\item
If $\lambda(d)=-1$ then there 
are infinitely many integer solutions $(x, n)$ to 
 $dx^2-Mn^2=1$.
\end{enumerate}
\end{cor}
\begin{proof}
For any square-free positive integer $d$, 
using the Chinese Remainder Theorem we may find integers $m_i$ and $e_i$ that satisfy all the 
conditions given in Theorem \ref{m}. To further ensure that these integers are primes, we may
use Dirichlet's theorem of primes in arithmetic progressions.

We apply  Theorem \ref{m} with $t=\pm 1$ in the first case above 
 and $t=-1$ for the second case (as required by the theorem when $\lambda(d)=-1$). 
Note that 
$M=me_1e_2$ and $\lambda(M)=-\lambda(d)$.

Observe that as $D=mde_1e_2$ has a prime 
divisor congruent to $3\pmod 4$, namely $e_1$, the norm of the fundamental unit 
is $1$. Therefore by Lemma 2.2 
one of the forms given in (\ref{normform}) is in the principal class and hence in the principal genus.  The result now follows 
from Theorem \ref{m} which states that the only forms in the principal genus are  
$(M, 0, -d)$ when $t=1$ and 
 $(d, 0, -M)$ when $t=-1$. The corollary now follows as all forms in the principal class
represent the integer $1$.
\end{proof}
\end{section}

\begin{section} {Sign changes for $\lambda(x^2\pm d)$}
In this section we will prove our main result, Theorem 1.3.

Suppose that $D=dt^2$, where $d$ is square-free. From Lemma 3.3 if $\lambda(n^2+d)$ changes sign 
infinitely often, then so does $\lambda(n^2+D)$. Thus we may assume that $d$ is square-free and positive
when we consider polynomials $x^2\pm d$. By Lemma 3.4 we may further assume that $d$ is not prime.

Consider $n^2+d$ where $d>0$. 
By Lemmas 3.1 and 3.2 we may assume that $\lambda(d)=1$.  From Corollary 3.6 we 
have an integer $M$ such that there are infinitely many solutions $(k, l)$ to 
$Mk^2-dl^2=1$ with $\lambda(M)=-\lambda(d)$. 
 Taking $n=dl$ we have 
$n^2+d=d(d l^2+1 )=dMk^2$. For $n^2-d$ with $d>0$, again Corollary 3.6 gives an integer 
$M$ with $\lambda(M)=-\lambda(d)$ such that 
there are infinitely many solutions $(k, l)$ to 
$Mk^2-dl^2=-1$. As above, if $n=dl$ then
$n^2-d=d(dl^2-1)=dMk^2$. 

 Therefore we have $\lambda(n^2\pm d)=\lambda(dM)=-1$ and hence $\lambda(n^2\pm d)$ takes
on the negative sign infinitely often and hence the result follows from Lemma 3.1.

\end{section}


\begin{thebibliography}{99}
\bibitem[BCG13]{BCG}{ P. Borwein, S.K.K. Choi and H. Ganguly},
 \textit{Sign changes of the Liouville function on quadratics},
 { Canad. Math. Bull.} {\bf 56} (2013), pp. 251 -- 257. 

\bibitem[CFMRS00]{C}{ J. Cassaigne, S. Ferenczi, C. Mauduit, 
J. Rivat and A. Sarkozy}
\textit{The Liouville function II},
Acta Arithmetica XCV. {\bf 4} (2000), 343--359.

\bibitem[C65]{Ch}{ S. Chowla}, 
\textit{The Riemann hypothesis and Hilbert's tenth problem}
 { Gordon and Breach, New York}, 1965.

\bibitem[MS12]{MS}{R. A. Mollin and A. Srinivasan,}
Pell equations: non-principal Lagrange criteria and central norms, Canadian Mathematical bulletin, {\bf 55} (2012), 774--782.


\bibitem[R00]{R}{ P. Ribenboim}, \textit{My Numbers, My Friends, Popular
Lectures on Number Theory}, { Springer-Verlag}, 2000.


\bibitem[S15]{S}{A. Srinivasan,} {\textit{Sign changes of the Liouville function on some irreducible quadratic polynomials}}, Journal of Combinatorics and Number Theory, {\bf 7.1} (2015).  

\bibitem[T20]{T}{J. Ter{\"{a}}v{\"{a}}inen},
\textit{On the Liouville function at polynomial arguments}, arXiv:2010.07924v2
 \end{thebibliography}
\end{document}